\documentclass[11pt]{article}
\usepackage[T1]{fontenc}
\usepackage{graphicx}
\usepackage{color}
\usepackage{amsmath,amssymb}
\usepackage{latexsym, amsthm}
\usepackage{enumerate}
\usepackage{bbm}
\usepackage{hyperref}
\usepackage{authblk}
\usepackage{setspace} 
\usepackage{accents}
\usepackage{ulem} 
\large\normalsize

\theoremstyle{plain}
\newtheorem{theorem}{Theorem}[section]

\newtheorem{lemma}[theorem]{Lemma}

\theoremstyle{definition}

\numberwithin{equation}{section}
\title{
Continuous dependence of an invariant measure on the jump rate of a piecewise-deterministic Markov process}
\author[1]{Dawid Czapla}
\author[2]{Sander C. Hille}
\author[1]{Katarzyna Horbacz}
\author[1]{Hanna~Wojew\'odka-\'Sci\k{a}\.zko}
\affil[1]{\small Institute of Mathematics, University of Silesia in Katowice, Bankowa 14, 40-007 Katowice, Poland}
\affil[2]{\small Mathematical Institute, Leiden University, P.O. Box 9512, 2300 RA Leiden, The Netherlands}
\date{}
\begin{document}
\maketitle
\begin{abstract}
We investigate a piecewise-deterministic Markov process, evolving on a Polish metric space, whose deterministic behaviour between random jumps is governed by some semi-flow, and any state right after the jump is attained by a randomly selected continuous transformation. It is assumed that the jumps appear at random moments, which coincide with the jump times of a Poisson process with intensity $\lambda$. The model of this type, although in a more general version, was examined in our previous papers, where we have shown, among others, that the Markov process under consideration possesses a unique invariant probability measure, say $\nu_{\lambda}^*$. 
The aim of this paper is to prove that the map $\lambda\mapsto\nu_{\lambda}^*$ is continuous (in the topology of weak convergence of probability measures). The studied dynamical system is inspired by certain stochastic models for cell division and gene expression.
\end{abstract}

\section*{Introduction}
Piecewise-deterministic Markov processes (PDMPs) originate with M.H.A. Davis \cite{davis}. They constitute an important class of Markov processes that is complementary to those defined by stochastic differential equations. PDMPs are encountered as suitable mathematical models for processes in the physical world around us, e.g. in resource allocation and service provisioning (queing, cf. \cite{davis}) or biology: as stochastic models for gene expression \cite{tyran}, cell division \cite{lm}, gene regulation \cite{hhs}, excitable membranes \cite{riedler_ea} or population dynamics \cite{alkurdi}. 

Mathematical research on PDMPs has been conducted over the years in various directions. Applications in control and optimization have been just one direction. The fundamentals of existence and uniqueness of invariant probability measures for Markov operators and semigroups associated to PDMPs, as well as their asymptotic properties, have attracted much attention. See e.g. \cite{costa2000,costa}, where the considered underlying state space is locally compact. The theory for the general case of non-locally compact Polish state space is less developed yet. It is considered e.g. in \cite{hhs,riedler_ea,dawid,asia,hw}. Another direction is that of establishing the validity of the Strong Law of Large Numbers (SLLN), the Central Limit Theorem (CLT) and the Law of the Interated Logarithm (LIL) for these non-stationary Markov processes (cf. \cite{hhsw,hhsw2,clt_chw,lil_chw}), which has interest in itself for non-stationary processes in general \cite{klo}.

In this paper, we are concerned with a special case of the PDMP described in \cite{dawid,asia}, whose deterministic motion between jumps depends on a~single continuous semi-flow, and any post-jump location is attained by a continuous transformation of the pre-jump state, randomly selected (with a place-dependent probability) among all possible ones. The jumps in this model occur at random time points according to a homogeneous Poisson process.
 The random dynamical system of this type constitutes a mathematical framework for certain particular biological models, such as those for gene expression \cite{tyran} or cell division \cite{lm}. 

The aim of the paper is to establish the continuous (in the Fortet-Mourier distance, cf. \cite[Section 8.3]{bogachev}) dependence of the invariant measure on the rate of the Poisson process determining the frequency of jumps. While the SLLN and the CLT provide the theoretical foundation for successful approximation of the invariant measure by means of observing or simulating (many) sample trajectories of the process, this result asserts the stability of this procedure, at least locally in parameter space. It is a~prerequisite for the development of a bifurcation theory. Moreover, even stronger regularity of this dependence on parameter (i.e. differentiability in a suitable norm on the space of measures) would be needed for applications in control theory or for parameter estimation (see e.g. \cite{hille_lyczek}).

The outline of the paper is as follows. In Section \ref{sec:1}, several facts on integrating measure-valued functions and basic definitions from the theory of Markov operators have been compiled. Section \ref{sec:model} deals with the structure and assumptions of the model under study. In Section \ref{sec:properties}, we establish certain auxiliary results on the transition operator of the Markov chain given by the post-jump locations. More specifically, we show that the operator is jointly continuous (in the topology of weak convergence of measures) as a function of measure and the jump-rate parameter, and that 
the weak convergence of the distributions of the chain to its unique stationary distribution must be uniform. Section \ref{sec:main} is the essential part of the paper. Here, we establish the announced results on the continuous dependence of the invariant measure on the jump frequency for both, the discrete-time system, constituted by the post jump-locations, and for the PDMP itself.

\section{Prelimenaries}\label{sec:1}
Let $X$ be a closed subset of some separable Banach space $(H,\|\cdot\|)$, endowed with the \hbox{$\sigma$-field} $\mathcal{B}_X$ consisting of its Borel subsets. Further, let $(BM(X),\|\cdot\|_{\infty})$ stand for the Banach space of all bounded Borel-measurable functions $f:X\to\mathbb{R}$ with the supremum norm $\|f\|_{\infty}:=\sup_{x\in X} |f(x)|.$ By $BC(X)$ and $BL(X)$ we shall denote the subspaces of $BM(X)$ consisting of all continuous and all Lipschitz continuous functions, respectively. Let us further introduce
\begin{align*}
\|f\|_{BL}:=\max\left\{\|f\|_{\infty}, |f|_{Lip}\right\}\;\;\;\text{for any}\;\;\;f\in BL(X),
\end{align*}
where 
\begin{align*}
|f|_{Lip}:=\sup\left\{\frac{|f(x)-f(y)|}{\|x-y\|}:\;x,y\in X,\,x\neq y  \right\}.
\end{align*}
It is well-known (cf. \cite[Proposition 1.6.2]{weaver}) that $\|\cdot\|_{BL}$ defines a norm in $BL(X)$, for which it is a~Banach space. 

In what follows, we will write $(\mathcal{M}_{sig}(X), \|\cdot\|_{TV})$ for the Banach space of all finite, countably additive functions (signed measures) on $\mathcal{B}_X$, endowed with the total variation norm $\|\cdot\|_{TV}$, which can be expressed~as
\begin{align*}
\|\mu\|_{TV}:=|\mu|(X)=\sup\left\{\left|\left\langle f,\mu\right\rangle\right| :\; f\in BM(X),\,\|f\|_{\infty}\leq 1\right\}\;\;\;\text{for}\;\;\;\mu\in\mathcal{M}_{sig}(X), 
\end{align*}
where
\begin{align*}
\left\langle f,\mu\right\rangle:=\int_X f(x)\mu(dx)
\end{align*}
and $|\mu|$ stands for the absolute variation of $\mu$. The symbols $\mathcal{M}_+(X)$ and $\mathcal{M}_1(X)$ will be used to denote the subsets of $\mathcal{M}_{sig}(X)$, consisting of all non-negative and all probability measures on $\mathcal{B}_X$, respectively. Moreover, we will write $\mathcal{M}_{1,1}(X)$ for the set of all measures $\mu\in\mathcal{M}_1(X)$ with finite first moment, i.e. satisfying $\langle\|\cdot\|,\mu\rangle<\infty$.

Let us now define, for any $\mu\in\mathcal{M}_{sig}(X)$, the linear functional $I_{\mu}:BL(X)\to \mathbb{R}$ given by 
\begin{align*}
I_{\mu}(f)=\langle f,\mu\rangle\;\;\;\text{for}\;\;\;f\in BL(X).
\end{align*}
It easy to show that $I_{\mu}\in BL(X)^*$ for every $\mu\in\mathcal{M}_{sig}(X)$, where $BL(X)^*$ stands for the dual space of $(BL(X), \|\cdot\|_{BL})$ with the operator norm $\|\cdot\|_{BL}^*$ given by
\begin{align*}
\|\varphi\|_{BL}^*:=\sup\left\{|\varphi(f)|:\; f\in BL(X),\,\|f\|_{BL}\leq 1\right\}\;\;\;\text{for any}\;\;\; \varphi\in BL(X)^*.
\end{align*}
Moreover, we have $\|I_{\mu}\|_{BL}^*\leq\|\mu\|_{TV}$ for any $\mu\in\mathcal{M}_{sig}(X)$.

Furthermore, it is well known (see \cite[Lemma~6]{dudley_baire}), that the mapping 
$$\mathcal{M}_{sig}(X)\ni\mu \mapsto I_{\mu}\in BL(X)^*$$ 
is injective, and thus the space $(\mathcal{M}_{sig}(X), \|\cdot\|_{TV})$ may be embedded into $(BL(X)^*, \|\cdot\|_{BL}^*)$. This enables us to identify each measure $\mu\in \mathcal{M}_{sig}(X)$ with the functional $I_{\mu}\in BL(X)^*$. Note that $\|\cdot\|_{BL}^*$ induces a norm on $\mathcal{M}_{sig}(X)$, called the Fortet-Mourier (or bounded Lipschitz) norm and denoted by $\|\cdot\|_{FM}$. Consequently, we can write
\begin{align*}
\|\mu\|_{FM}:=\left\|I_{\mu}\right\|_{BL}^*=\sup\{|\langle f,\mu\rangle|:\,f\in BL(X),\,\|f\|_{BL}\leq 1\}\;\;\;\text{for any}\;\;\; \mu\in\mathcal{M}_{sig}(X).
\end{align*}
As we have already seen, generally $\|\mu\|_{FM}=\|I_{\mu}\|_{BL}^* \leq \|\mu\|_{TV}$ for any $\mu\in\mathcal{M}_{sig}(X)$. However, for positive measures the norms coincide, i.e. $\|\mu\|_{FM}=\mu(X)=\|\mu\|_{TV}$ for all $\mu\in\mathcal{M}_+(X)$.

Let us now write $\mathcal{D}(X)$ and $\mathcal{D}_+(X)$ for the linear space and the convex cone, respectively, generated by the set $\{\delta_x:\,x\in X\}\subset BL(X)^*$ of functionals of the form 
\begin{align*}
\delta_x(f):=f(x)\;\;\;\text{for any}\;\;\;f\in BL(X),\;x\in X,
\end{align*} 
which can be also viewed as Dirac measures. 
It is not hard to check that the \hbox{$\|\cdot\|_{BL}^*$-closure} of $\mathcal{D}(X)$ is a~separable Banach subspace of $BL(X)^*$. Moreover, assuming that $X$ is complete, one can show 
that $\mathcal{M}_+(X)=\text{cl}\,\mathcal{D}_+(X)$ \hbox{(cf. \cite[Theorems 2.3.8--2.3.19]{worm})}, which in turn implies that $\mathcal{M}_{sig}(X)$ is  a~\hbox{$\|\cdot\|_{BL}^*$-dense} subspace of $\text{cl}\,\mathcal{D}(X)$, i.e. \hbox{$\text{cl}\,\mathcal{M}_{sig}(X)=\text{cl}\,\mathcal{D}(X)$}. The key idea underlying the proof of this result is to show that every measure $\mu\in\mathcal{M}_+(X)$ can be represented by the Bochner integral (for details see e.g. \cite{vector_measures}) of the continuous map $X\ni x\mapsto \delta_x\in \text{cl}\,\mathcal{D}(X)$, i.e.
\begin{align*}
\mu=\int_X \delta_x\,\mu(dx)\in \text{cl}\, \mathcal{D}_+(X).
\end{align*}
In particular, it follows that $\left(\text{cl}\,\mathcal{M}_{sig}(X),\|\cdot\|_{BL}^*|_{\text{cl}\,\mathcal{D}(X)}\right)$ is a separable Banach space.

What is more, according to \cite[Theorem 2.3.22]{worm}, the dual space of \hbox{$\text{cl}\,\mathcal{M}_{sig}(X)=\text{cl}\,\mathcal{D}(X)$} with the operator norm
\begin{align*}
\|\kappa\|_{\text{cl}\,\mathcal{D}}^{**}:=\sup\{|\kappa(\varphi)|:\,\varphi \in \text{cl}\,\mathcal{D}(X), \; \|\varphi\|_{BL}^*\leq 1\},\;\;\; \kappa\in[\text{cl}\,\mathcal{D}(X)]^*,
\end{align*}
is isometrically isomorphic with the space $(BL(X),\|\cdot\|_{BL})$, and each functional $\kappa\in [\text{cl}\,\mathcal{D}(X)]^*$ can be represented by some $f\in BL(X)$, in the sense that $\kappa(\varphi)=\varphi(f)$ for $\varphi\in \text{cl}\,\mathcal{D}(X)$. In particular, we then have 
$\kappa(\mu)=I_{\mu}(f)=\langle f,\mu\rangle$,  
whenever $\mu\in\mathcal{M}_{sig}(X)$ (by identyfing $\mu$ with $I_{\mu}$). 
 
In view of the above observations, the norm $\|\cdot\|_{BL}^*$ is convenient for integrating (in the Bochner sense) measure-valued functions $p: E\to\mathcal{M}_{sig}(X)$, where $E$ is an arbitrary measure space. The Pettis measurability theorem (see e.g. \cite[Chapter II, Theorem 2]{vector_measures}), together with the separability of $\text{cl}\,\mathcal{M}_{sig}(X)$, ensures that $p$ is strongly measurable as a map with values in $\text{cl}\,\mathcal{M}_{sig}(X)$ (i.e.~it is a pointwise a.e. limit of simple functions) if and only if, for any $f\in BL(X)$, the functional \hbox{$E\ni t\mapsto\langle f,p(t)\rangle\in\mathbb{R}$} is measurable. Moreover, we have at our disposal the following result (see \cite[Propositions 3.2.3-3.2.5]{worm} or \cite[Proposition C.2]{hille_evers}), which provides a tractable condition guaranteeing the integrability of $p$ and ensuring that the integral is an element of $\mathcal{M}_{sig}(X)$:

\begin{theorem}\label{thm:b_int}
Let $(E,\Sigma)$ be a measurable space with a $\sigma$-finite measure $\nu$, and let \linebreak\hbox{$p:E \to\mathcal{M}_{sig}(X)$} be a strongly measurable function. Suppose that there exists a real-valued function $g\in\mathcal{L}^1(E,\Sigma,\nu)$ such that 
\begin{align*}
\|p(t)\|_{TV}\leq g(t)\;\;\;\text{for a.e.}\;\;\; t\in E.
\end{align*}
Then then the following conditions holds:
\begin{itemize}
\item[(i)] The function $p$ is Bochner $\nu$-integrable as a map acting from $(E,\Sigma)$ to \linebreak$\left(\text{cl}\,\mathcal{M}_{sig}(X),\|\cdot\|_{BL}^*|_{\text{cl}\,\mathcal{D}(X)}\right)$. Moreover, we have
$$\left\|\int_E p(t)\,\nu(dt)\right\|_{TV}\leq \int_E \|p(t)\|_{TV}\,\nu(dt).$$
\item[(ii)] The Bochner integral $\int_E p(t)\,\nu(dt)\in\text{cl}\,\mathcal{M}
_{sig}(X)$ belongs to~$\mathcal{M}
_{sig}(X)$ and satisfies
\begin{align*}
\left(\int_E p(t)\nu(dt)\right)(A)=\int_E p(t)(A)\nu(dt)\;\;\;\text{for any}\;\;\; A\in \mathcal{B}_{X}.
\end{align*}
\end{itemize}
\end{theorem}

Another crucial observation is that the restriction of the weak topology on $\mathcal{M}
_{sig}(X)$, generated by $BC(X)$, to $\mathcal{M}_+(X)$ equals to the topology induced by the norm \linebreak\hbox{$\|\cdot\|_{BL}^*|_{\mathcal{M}_+(X)}=\|\cdot\|_{FM}|_{\mathcal{M}_+(X)}$} \hbox{(cf. \cite[Theorem 18]{dudley_baire} or \cite[Theorem 8.3.2]{bogachev})}. In particular, the following holds:
\begin{theorem}
Let $\mu_n,\mu\in\mathcal{M}_+(X)$ for every $n\in\mathbb{N}$. Then $\lim_{n\to\infty} \|\mu_n-\mu\|_{FM}=0$ if and only if $\mu_n\stackrel{w}{\to}\mu$, that is, 
$$\lim_{n\to\infty}\langle f,\mu_n\rangle=\left\langle f,\mu\right\rangle\;\;\; \text{for any}\;\;\; f\in BC(X).$$
\end{theorem}

Let us now recall several basic definitions concerning Markov chains. First of all, a~function \hbox{$P:X\times\mathcal{B}_X\to[0,1]$} is called a~stochastic kernel if, for any fixed $A\in \mathcal{B}_X$, \hbox{$x\mapsto P(x,A)$} is a Borel-measurable map on $X$, and, for any fixed $x\in X$, $A\mapsto P(x,A)$ is a~probability Borel measure on $\mathcal{B}_X$. We can consider two operators corresponding to a stochastic kernel $P$, namely
\begin{equation} \label{def:markov} 
\mu P(A)=\int_{X} P(x,A)\,\mu(dx)\;\;\;\text{for}\;\;\; \mu\in\mathcal{M}_{sig}(X),\;A\in \mathcal{B}_X
\end{equation}
and
\begin{equation}\label{def:dual} 
Pf(x)=\int_{X} f(y)\,P(x,dy)\;\;\;\text{for}\;\;\;f\in BM(X),\; x\in X. 
\end{equation}
The operator $(\cdot)P:\mathcal{M}_{sig}(X) \to \mathcal{M}_{sig}(X)$, given by \eqref{def:markov}, is called a~regular Markov operator. It is easy to check that
\begin{equation*}
\langle f, \mu P\rangle=\langle Pf,\mu\rangle\;\;\;\mbox{for any}\;\;\; f\in BM(X),\;\mu\in\mathcal{M}_{sig}(X),
\end{equation*}
and, therefore, $P(\cdot):BM(X)\to BM(X)$, defined by \eqref{def:dual}, is said to be the dual operator of $(\cdot)P$. 

A regular Markov operator $(\cdot)P$ is said to be Feller if its dual operator $P(\cdot)$ preserves continuity, that is, $Pf\in BC(X)$ for every $f\in BC(X)$. A measure $\mu^*\in\mathcal{M}_+(X)$ is called an invariant measure for~$(\cdot)P$ whenever $\mu^*P=\mu^*$. 

We will say that the operator $(\cdot)P$ is exponentially ergodic in the Fortet-Mourier distance if there exists a unique measure $\mu^*\in\mathcal{M}_1(X)$ invariant of $(\cdot)P$, for which there is $q\in [0,1)$ such that, for any $\mu\in\mathcal{M}_{1,1}(X)$ and some constant $C(\mu)\in\mathbb{R}_+$, we have
$$\left\|\mu P^n-\mu^*\right\|_{FM}\leq C(\mu)q^n\;\;\;\text{for any}\;\;\; n\in\mathbb{N}.$$
The measure $\mu^*$ is then usually called exponentially attracting.

A regular Markov semigroup $({P}(t))_{t\in\mathbb{R}_+}$ is a family of regular Markov operators  \linebreak\hbox{$(\cdot){P}(t): \mathcal{M}_{sig}(X) \to \mathcal{M}_{sig}(X)$}, $t\in\mathbb{R}_+:=[0,\infty)$, which form a semigroup (under composition) with the identity transformation $(\cdot){P}(0)$ as the unity element. 
Provided that $(\cdot){P}(t)$ is a Markov-Feller operator for every $t\in\mathbb{R}_+$, the semigroup $({P}(t))_{t\in\mathbb{R}_+}$ is said to be Markov-Feller, too. If, for some $\mu^*\in\mathcal{M}_{fin}(X)$, $\mu^*{P}(t)=\mu^*$ for every $t\in\mathbb{R}_+$, then we call $\mu^*$ an invariant measure of $({P}(t))_{t\in\mathbb{R}_+}$.

\section{Description of the model}\label{sec:model}

Recall that $X$ is a closed subset of some separable Banach space $(H,\|\cdot\|)$, and let $(\Theta,\mathcal{B}_{\Theta},\vartheta)$ be a~topological measure space with a $\sigma$-finite Borel measure~$\vartheta$. With a slight abuse of notation, we will further write $d\theta$ only, instead of $\vartheta(d\theta)$.

Let us consider a PDMP $({X}(t))_{t\in\mathbb{R}_+}$, evolving on the space $X$ through random jumps occuring at the jump times $\tau_n$, $n\in\mathbb{N}$, of a homogeneous Poisson process with intensity $\lambda>0$. The state right after the jump is attained by a transformation $w_{\theta}:X\to X$, randomly selected from the set $\{w_{\theta}:\theta\in\Theta\}$. The probability of choosing $w_{\theta}$ is determined by a~place-dependant density function $\theta \mapsto p(x,\theta)$, where $x$ describes the state of the process just before the jump. It is required that the maps \hbox{$(x,\theta)\mapsto p(x,\theta)$} and $(x,\theta)\mapsto w_{\theta}(x)$ are continuous. Between the jumps, the process is deterministically driven by a~continuous (with respect to each variable) semi-flow $S:\mathbb{R}_+\times X\to X$. The flow property means, as usual, that
$S(0,x)=x$ and $S(s+t,x)=S(s,S(t,x))$ for any $x\in X$ and any $s,t\in\mathbb{R}_+$.

Let us now move on to the formal description of the model. 
For any $\mu\in\mathcal{M}_1(X)$ and any  $\lambda>0$ we first define, on some suitable probability space $(\Omega, \mathcal{F}, \mathbb{P}_{\mu})$, a stochastic proces $(X_n)_{n\in\mathbb{N}_0}$ with initial distribution $\mu$, by setting
\begin{align*}
X_{n+1}=w_{\theta_{n+1}}\left(S\left(\Delta\tau_{n+1},X_{n}\right)\right)\;\;\;\text{for}\;\;\; n\in\mathbb{N}_0,
\end{align*}
with $\Delta\tau_{n+1}=\tau_{n+1}-\tau_n$, where $(\tau_n)_{n\in\mathbb{N}_0}$ and $(\theta_n)_{n\in\mathbb{N}}$ are sequences of random variables with values in $\mathbb{R}_+$ and $\Theta$, respectively, defined in such a way that $\tau_0=0$,  $\tau_n \to \infty$ $\mathbb{P}_{\mu}$-a.s., as $n\to\infty$, and 
\begin{gather*}
\mathbb{P}_{\mu} \left(\Delta \tau_{n+1} \leq t \,|\, W_n\right)=1-e^{-\lambda t}\;\;\;\text{for any}\;\;\;t\in\mathbb{R}_+,\\
\mathbb{P}_{\mu}\left(\theta_{n+1}\in B\,|\,S\left(\Delta\tau_{n+1},X_n\right)=x,\,W_n\right)=\int_Bp(x,\theta)\,d\theta\;\;\;\text{for any}\;\;\;x\in X,\;B\in\mathcal{B}_{\Theta},
\end{gather*} 
with $W_0:=X_0$ and  $W_n:=(W_0,\tau_1,\ldots,\tau_n,\theta_1,\ldots,\theta_n)$ for $n\in\mathbb{N}$. We also demand that, for any $n\in\mathbb{N}_0$, the variables $\Delta\tau_{n+1}$ and $\theta_{n+1}$ are conditionally independent given $W_n$.

A standard computation shows that, for any $\lambda>0$, $(X_{n})_{n\in\mathbb{N}_0}$ is a time-homogeneous Markov chain with transition law $P_{\lambda}:X\times \mathcal{B}_X\to[0,1]$ given by 
\begin{align}\label{def:Pi_lambda}
P_{\lambda}(x,A)=\int_0^{\infty}\lambda e^{-\lambda t}\int_{\Theta}p(S(t,x),\theta)\,\mathbbm{1}_A\left(w_{\theta}(S(t,x))\right)\,d\theta\,dt\;\;\;\text{for}\;\;\;x\in X,\;A\in\mathcal{B}_X,
\end{align}
that is, 
\begin{align*}P_{\lambda}(x,A)=P\left(X_{n+1}\in A\, | \, X_n=x\right)  \;\;\; \text{for any} \;\;\;x\in X,\;A\in\mathcal{B}_X.
\end{align*}

On the same probability space, we now define a Markov process $({X}(t))_{t\in\mathbb{R}_+}$, as an  iterpolation of the chain $(X_n)_{n\in\mathbb{N}_0}$, namely
\begin{align*}
{X}(t)=S\left(t-\tau_n,X_n\right)\;\;\;\text{for}\;\;\;t\in[\tau_n,\tau_{n+1}),\;\;\;n\in\mathbb{N}_0.
\end{align*} 
By $( {P}_{\lambda}(t))_{t\in\mathbb{R}_+}$ we shall denote the Markov semigroup associated with the process $\left({X}(t)\right)_{t\in\mathbb{R}_+}$, so that, for any $t\in\mathbb{R}_+$, ${P}_{\lambda}(t)$ is the Markov operator corresponding to the stochastic kernel satisfying
\begin{align}\label{def:P(t)}
{P}_{\lambda}(t)(x,A)= \mathbb{P}_{\mu}\left({X}(s+t)\in A\, | \, {X}(s)=x\right)  \;\;\;\text{for any}\;\;\; A\in\mathcal{B}_X,\; x\in X,\; s \in\mathbb{R}_+.
\end{align}

We further assume that there exist a point $\bar{x}\in X$, a Borel measurable function \hbox{$J:X\to[0,\infty)$} and constants \hbox{$\alpha\in\mathbb{R}$, $L,L_w,L_p,\lambda_{\min},\lambda_{\max},\overline{p}>0$}, such that
\begin{align}\label{constants_ass}
LL_w+\frac{\alpha}{\lambda}<1\;\;\;\text{for each}\;\;\;\lambda\in [\lambda_{\min},\lambda_{\max}],
\end{align}
and, for any $x,y\in X$, the following conditions hold:
\begin{gather}
\sup_{x\in X}\int_0^{\infty}   e^{-\lambda_{\min} t}\int_{\Theta} p\left(S(t,x),\theta\right)
\left\|w_{\theta} \left(S(t,\bar{x})\right)\right\|\,d\theta\,dt<\infty,
\label{A1}\\
\left\|S(t,x)-S(t,y)\right\|\leq Le^{\alpha t}\|x-y\|\;\;\;\text{for}\;\;\;t\in\mathbb{R}_+,
\label{A2}\\
\left\|S(t,x)-S(s,x)\right\|\leq\left\{
\begin{array}{ll}
(t-s)e^{\alpha s}J(x),&\text{if}\;\;\alpha\leq 0\\
(t-s)e^{\alpha t}J(x),&\text{if}\;\;\alpha >0
\end{array}
\right.\;\;\;\text{for}\;\;\;0\leq s\leq t,
\label{A6}\\
\int_{\Theta} p(x,\theta)\,\left\|w_{\theta}(x)-w_{\theta}(y)\right\|\,d\theta\leq L_w\left\|x-y\right\|,
\label{A3}\\
\int_{\Theta} |p(x,\theta)-p(y,\theta)|\,d\theta\leq L_{p} \left\|x-y\right\|,
\label{A4}\\
\begin{split}
&\int_{\Theta(x,y)}\min\{p(x,\theta),p(y,\theta)\}\,d\theta\geq \overline{p}, 
\;\;\;\text{where}\\
&\Theta(x,y):=\{\theta\in\Theta:\, 
\left\|w_{\theta}(x)-w_{\theta}(y)\right\|\leq L_w \left\|x-y\right\|\}.
\end{split}\label{A5}
\end{gather}

Note that, if $(H,\left\langle\cdot|\cdot\right\rangle)$ is a Hilbert space and $A:X\to H$ is an $\alpha$-dissipative operator with $\alpha\leq 0$, i.e.
\begin{align*}
\left\langle Ax-Ay|x-y\right\rangle\leq \alpha\left\|x-y\right\|^2\;\;\;\mbox{for any}\;\;\;x,y\in X,
\end{align*}
which additionally satisfies the so-called range condition, that is, for some $T>0$, 
\begin{align*}
X\subset \text{Range}\left(\operatorname{id}_X-tA\right)\;\;\;\text{for}\;\;\; t\in(0,T),
\end{align*}
then, for any $x\in X$, the Cauchy problem of the form
\begin{align*}
\left\{\begin{array}{l}
y'(t)=A(y(t))\\
y(0)=x
\end{array}\right.
\end{align*}
has a unique solution $t\mapsto S(t,x)$ such that the semi-flow $S$ enjoys conditions \eqref{A2}, with $L=1$, and   \eqref{A6} (cf. \cite[Theorem 5.3, Corollary 5.4]{cl} and \cite[Section 3]{dawid}).

Note that, upon assuming \eqref{constants_ass}, we have $\lambda>\max\{0,\alpha\}$ for any $\lambda\in[\lambda_{\min},\lambda_{\max}]$. Let us further write shortly 
\begin{align}\label{def:bar_alpha}
\bar{\alpha}:=\max\{0,\alpha\}.
\end{align}

\section{Some proerties of the operator $P_{\lambda}$}\label{sec:properties}

Consider the abstract model given in Section \ref{sec:model}. In order to simplify notation, for any $t\in\mathbb{R}_+$, let us introduce the function $\Pi_{(t)}:X\times \mathcal{B}_X\to[0,1]$ given by
\begin{align}\label{def:P_phi_t}
\Pi_{(t)}(x,A):=\int_{\Theta}p\left(S(t,x),\theta\right)\,\mathbbm{1}_A\left(w_{\theta}\left(S(t,x)\right)\right)\,d\theta
\;\;\;\text{for}\;\;\;x\in X,\;A\in\mathcal{B}_X.
\end{align}
Note that $\Pi_{(t)}$ is a stochastic kernel, and that the corresponding Markov operator is Feller, due to the continuity of $p$, $S$ and $w_{\theta}$, $\theta\in\Theta$. 
Moreover, observe that, for an arbitrary $\lambda>0$, we have 
\begin{align}
\begin{split}\label{eq:P_lambda_P_phi}
\mu P_{\lambda}(A)&=\int_X \int_0^{\infty} \lambda e^{-\lambda t} \Pi_{(t)}(x,A)\,dt\,\mu(dx)=\int_0^{\infty}\lambda e^{-\lambda t}\int_X \Pi_{(t)}(x,A) \,\mu(dx)\,dt\\
&=\int_0^{\infty} \lambda e^{-\lambda t} \mu\Pi_{(t)}(A)\,dt\;\;\;\text{for any}\;\;\; \mu\in\mathcal{M}_{sig}(X),\;A\in \mathcal{B}_X.
\end{split}
\end{align}

\begin{lemma}\label{lem:Bochner}
Suppose that conditions \eqref{A6}-\eqref{A4} hold with constants satisfying \eqref{constants_ass}. Then, for any $\lambda>0$ and any $\mu\in\mathcal{M}_{sig}(X)$ satisfying $\langle J,|\mu|\rangle<\infty$, where $J$ is given in \eqref{A6}, the map $t\mapsto e^{-\lambda t}\mu \Pi_{(t)}$ is Bochner integrable on $\mathbb{R}_+$, and we have
\begin{align*}
\mu P_{\lambda}=\int_0^{\infty}\lambda e^{-\lambda t} \mu \Pi_{(t)}\,dt.
\end{align*}
\end{lemma}

\begin{proof}
Let $\mu\in\mathcal{M}_{sig}(X)$ and $\lambda>0$. Note that condition \eqref{A6} implies that
\begin{align*}
\left\|S(t,x)-S(s,x)\right\|\leq J(x)e^{\bar{\alpha}(t+s)}|t-s|\;\;\;\text{for any}\;\;\; s,t\in\mathbb{R}_+,\;x\in X,
\end{align*}
where $\bar{\alpha}$ is given by \eqref{def:bar_alpha}. 
Hence, applying \eqref{A3} and \eqref{A4}, we see that,  for every \hbox{$f\in BL(X)$},
\begin{align*}
\left|\left\langle f,\mu\Pi_{(t)}\right\rangle-\left\langle f,\mu\Pi_{(s)}\right\rangle\right|
=&\left|\left\langle \Pi_{(t)}f-\Pi_{(s)}f,\mu\right\rangle\right|\\
\leq &\int_X \int_{\Theta}p(S(t,x),\theta)\,\left|f(w_{\theta}(S(t,x)))-f(w_{\theta}(S(s,x)))\right|\,d\theta\,|\mu|(dx)\\
&+\int_X\int_{\Theta}\left|p(S(t,x),\theta)-p(S(s,x),\theta)\right|\,\left|f(w_{\theta}(S(s,x)))\right|\,d\theta\,|\mu|(dx)\\
\leq &\left(|f|_{Lip}L_w+\|f\|_{\infty}L_p\right)\int_X \left\|S(t,x)-S(s,x)\right\|\,|\mu|(dx)\\
\leq &\|f\|_{BL}\left(L_w+L_p\right)\left\langle J,|\mu|\right\rangle e^{\bar{\alpha}(t+s)}|t-s|\;\;\;\text{for}\;\;\;s,t\in\mathbb{R}_+.
\end{align*}
This shows that the map $t\mapsto \left\langle f,e^{-\lambda t}\mu \Pi_{(t)}\right\rangle$ is continuous for any $f\in BL(X)$, and thus it is Borel measurable. Consequently, it now follows from the Pettis theorem (cf. \cite{vector_measures}) that  the map $t\mapsto e^{-\lambda t}\mu\Pi_{(t)}$ is strongly measurable. Furthermore, we have
\begin{align*}
\left\|e^{-\lambda t}\mu\Pi_{(t)}\right\|_{TV}
\leq 
\left\|\mu\right\|_{TV}e^{-\lambda t}\;\;\;\text{for any}\;\;\; t\in\mathbb{R}_+,
\end{align*}
which, due to Theorem \ref{thm:b_int}, yields that $t\mapsto e^{-\lambda t}\mu\Pi_{(t)}\in\text{cl}\,\mathcal{M}_{sig}(X)$ is Bochner integrable (with respect to the Lebesgue measure) on $\mathbb{R}_+$, and that the integral is a measure in $\mathcal{M}_{sig}(X)$, which satisfies
\begin{align*}
\left(\int_0^{\infty} \lambda e^{-\lambda t} \mu\Pi_{(t)}\,dt\right)(A)=\int_0^{\infty} \lambda e^{-\lambda t} \mu\Pi_{(t)}(A)\,dt\;\;\;\text{for any}\;\;\; A\in \mathcal{B}_X.
\end{align*}
The assertion of the lemma now follows from \eqref{eq:P_lambda_P_phi}.
\end{proof}

\begin{lemma}\label{lem:equicnt}
Let $f\in BL(X)$. Upon assuming \eqref{A2}, \eqref{A3} and \eqref{A4} with constants satisfying \eqref{constants_ass}, we have
\begin{align*}
\left\|\mu\Pi_{(t)}\right\|_{FM}\leq \left(1+\left(L_w+L_p\right)Le^{\alpha t}\right)\left\|\mu\right\|_{FM}\;\;\;\text{for any}\;\;\;\mu\in\mathcal{M}_{sig}(X),\;t\in\mathbb{R}_+.
\end{align*}
\end{lemma}

\begin{proof}
Let $f\in BL(X)$ be such that $\|f\|_{BL}\leq 1$. 
Obviously, 
$\|\Pi_{(t)}f\|_{\infty}\leq 1$ 
for every $t\in\mathbb{R}_+$. 
Moreover, from conditions  \eqref{A2}, \eqref{A3}, \eqref{A4} it follows that $\Pi_{(t)}f\in BL(X)$, and
\begin{align*}
|\Pi_{(t)}f|_{Lip}\leq (L_w+L_p)L e^{\alpha t}\;\;\;\text{for any}\;\;\;t\in\mathbb{R}_+,
\end{align*}
since 
\begin{align*}
\begin{aligned}
&\left|\Pi_{(t)}f(x)-\Pi_{(t)}f(y)\right|\\
&\qquad=\left|\int_{\Theta}p\left(S(t,x),\theta\right)f\left(w_{\theta}\left(S(t,x)\right)\right)d\theta-\int_{\Theta}p\left(S(t,y),\theta\right)f\left(w_{\theta}\left(S(t,y)\right)\right)d\theta\right|\\
&\qquad\leq 
\left(L_w+L_p\right)\|S(t,x)-S(t,y)\|\\
&\qquad\leq \left(L_w+L_p\right)Le^{\alpha t}\|x-y\|\;\;\;\text{for all}\;\;\;x,y\in X,\;t\in\mathbb{R}_+.
\end{aligned}
\end{align*}
Therefore, for any $\mu\in\mathcal{M}_{sig}(X)$ and any $t\in\mathbb{R}_+$, we obtain
\begin{align*}
\begin{aligned}
\left|\left\langle f,\mu\Pi_{(t)}\right\rangle\right|
=\left|\left\langle \Pi_{(t)}f,\mu\right\rangle\right|
\leq \left\|\Pi_{(t)}f\right\|_{BL}\left\|\mu\right\|_{FM},
\end{aligned}
\end{align*}
which gives the desired conclusion.
\end{proof}

\begin{lemma}\label{lem:mvt}
For any $\lambda_1,\lambda_2>0$, we have
\begin{align*}
\int_0^{\infty}\left|\lambda_1 e^{-\lambda_1 t}-\lambda_2 e^{-\lambda_2 t}\right|dt
\leq \left|\lambda_1-\lambda_2\right|\left(\frac{1}{\lambda_1}+\frac{1}{\lambda_2}\right).
\end{align*}
\end{lemma}

\begin{proof}
Without loss of generality, we may assume that $\lambda_1<\lambda_2$. 
Since $1-e^{-x}\leq x$ for every $x\in\mathbb{R}$, we obtain
\begin{align*}
\int_0^{\infty}\left|\lambda_1 e^{-\lambda_1 t}-\lambda_2 e^{-\lambda_2 t}\right|dt
&\leq \lambda_1\int_0^{\infty}\left|e^{-\lambda_1 t}-e^{-\lambda_2 t}\right|dt+\left(\lambda_2-\lambda_1\right)\int_0^{\infty}e^{-\lambda_2 t}dt\\
&\leq \lambda_1\int_0^{\infty}e^{-\lambda_1t}\left(1-e^{-(\lambda_2-\lambda_1)t}\right)dt+\frac{\lambda_2-\lambda_1}{\lambda_2}\\
&\leq \lambda_1\left(\lambda_2-\lambda_1\right)\int_0^{\infty}e^{-\lambda_1 t}t\,dt+\frac{\left(\lambda_2-\lambda_1\right)}{\lambda_2}\\
&= \left|\lambda_1-\lambda_2\right|\left(\frac{1}{\lambda_1}+\frac{1}{\lambda_2}\right),
\end{align*}
which completes the proof.
\end{proof}

\begin{lemma}\label{lem:jointly_cnt}
Let $\mathcal{M}_{sig}(X)$ be endowed with the topology induced by the norm $\|\cdot\|_{FM}$, and suppose that conditions \eqref{A2}-\eqref{A4} hold with constants satisfying \eqref{constants_ass}. Then, the map 
\hbox{$(\bar{\alpha},\infty)\times \mathcal{M}_{sig}(X)\ni(\lambda,\mu)\mapsto \mu P_{\lambda}\in\mathcal{M}_{sig}(X),$} 
where $\bar{\alpha}$ is given by \eqref{def:bar_alpha}, 
is jointly continuous.
\end{lemma}

\begin{proof}
Let $\lambda_1,\lambda_2>\bar{\alpha}$ and $\mu_1,\mu_2\in\mathcal{M}_{sig}(X)$. Note that, due to Lemma \ref{lem:Bochner}, we have
\begin{align*}
&\left\|\mu_1 P_{\lambda_1}-\mu_2 P_{\lambda_2}\right\|_{FM}
=\left\|\int_0^{\infty}\left(\lambda_1 e^{-\lambda_1 t}\mu_1\Pi_{(t)}-\lambda_2 e^{-\lambda_2 t}\mu_2\Pi_{(t)}\right)dt\right\|_{FM}\\
&\qquad\qquad\leq \left\|\mu_1\right\|_{TV}\int_0^{\infty}\left|\lambda_1 e^{-\lambda_1 t}-\lambda_2 e^{-\lambda_2 t}\right|\,dt+\int_0^{\infty}\lambda_2 e^{-\lambda_2 t}\left\|\mu_1\Pi_{(t)}-\mu_2\Pi_{(t)}\right\|_{FM}dt,
\end{align*}
where the inequality follows from statement (i) of Theorem \ref{thm:b_int} and the fact that \linebreak\hbox{$\|\mu_1\Pi_{(t)}\|_{FM}\leq\|\mu_1\|_{TV}$}. 
Further, applying Lemmas \ref{lem:equicnt} and \ref{lem:mvt}, we obtain
\begin{align*}
\begin{aligned}
\left\|\mu_1 P_{\lambda_1}-\mu_2 P_{\lambda_2}\right\|_{FM}
\leq\,\left\|\mu_1\right\|_{TV}\left|\lambda_1-\lambda_2\right|\left(\frac{1}{\lambda_1}+\frac{1}{\lambda_2}\right)
+\left\|\mu_1-\mu_2\right\|_{FM}
\left(1+\frac{\left(L_w+L_p\right)L\lambda_2}{\lambda_2-\alpha}\right).
\end{aligned}
\end{align*}
We now see that $\|\mu_1 P_{\lambda_1} - \mu_2 P_{\lambda_2} \|_{FM} \to 0$, as $|\lambda_1-\lambda_2|\to 0$ and $\|\mu_1-\mu_2\|_{FM}\to 0$, which completes the proof.
\end{proof}

Suppose that \eqref{A1}, \eqref{A2} and \eqref{A3}-\eqref{A5} hold with constants satisfying \eqref{constants_ass}. 
Then, according to \hbox{\cite[Theorem 4.1]{dawid}} (or \hbox{\cite[Theorem 4.1]{asia}}), for any $\lambda\in[\lambda_{\min},\lambda_{\max}]$, 
there exists a~unique invariant measure $\mu_{\lambda}^*\in\mathcal{M}_1(X)$ for $P_{\lambda}$ such that
\begin{align}\label{eq:conv_lambda}
\left\|\mu P^n_{\lambda}-\mu_{\lambda}^*\right\|_{FM}\leq C_{\lambda,\mu} \,q_{\lambda}^n\;\;\;\text{for any}\;\;\;\mu\in\mathcal{M}_{1,1}(X),
\end{align}
where $q_{\lambda}\in(0,1)$ and $C_{\lambda,\mu}\in\mathbb{R}_+$ are some constants,  depending on the parameter $\lambda$ and the initial measure~$\mu$.

\begin{lemma}\label{lem:uniform_conv}
Suppose that conditions \eqref{A1}, \eqref{A2} and \eqref{A3}-\eqref{A5} hold with constants satisfying \eqref{constants_ass}, and, for any $\lambda \in [\lambda_{min}, \lambda_{max}]$, let $\mu_{\lambda}^*$ stand for the unique invariant probability measure measure of $P_{\lambda}$. Then, the convergence $\|\mu P^n_{\lambda} - \mu_{\lambda}^*\|_{FM}\to 0$ (as $n\to \infty$) is uniform with respect to $\lambda$, whenever \hbox{$\mu\in\mathcal{M}_{1,1}(X)$}.
\end{lemma}

\begin{proof}
In view of \cite[Theorem 4.1]{dawid}, 
it is sufficient to prove that the convergence is uniform with respect to $\lambda$.

Let us consider the case where $\alpha\leq 0$. 
Choose an arbitrary $\lambda\in[\lambda_{\min},\lambda_{\max}]$, and note that, substituting $t=\lambda_{\max}\lambda^{-1} u$, we obtain
\begin{align*}
\mu P_{\lambda}(A)&=\int_X\int_0^{\infty}\lambda e^{-\lambda t}\int_{\Theta}p\left(S(t,x),\theta\right)\,\mathbbm{1}_A\left(w_{\theta}\left(S(t,x)\right)\right)\,d\theta\,dt\,\mu(dx)\\
&=\int_X\int_0^{\infty}\lambda_{\max} e^{-\lambda_{\max} u}\int_{\Theta}p\left(S_{\lambda}(u,x),\theta\right)\,\mathbbm{1}_A\left(w_{\theta}\left(S_{\lambda}(u,x)\right)\right)\,d\theta\,du\,\mu(dx)
\end{align*}
for any $\mu\in\mathcal{M}_1(X)$, $A\in\mathcal{B}_X$, where 
\begin{align*}
S_{\lambda}(u,x)=S\left(\frac{\lambda_{\max}}{\lambda} u,x\right)\;\;\;\text{for}\;\;\;u\in\mathbb{R}_+,\;x\in X.
\end{align*}
Moreover, the semi-flow $S_{\lambda}$ enjoys condition \eqref{A2}, since, for any $t\in\mathbb{R}_+$ and any $x,y\in X$, we have
\begin{align*}
\left\|S_{\lambda}(t,x)-S_{\lambda}(t,y)\right\|\leq Le^{\alpha\lambda_{\max}\lambda^{-1} t}\left\|x-y\right\|\leq L e^{\alpha t}\left\|x-y\right\|.
\end{align*}
Hence, we can write $P_{\lambda}=\widetilde{P}_{\lambda_{\max}}$, where $\widetilde{P}_{\lambda_{\max}}$ stands for the Markov operator corresponding to the instance of our system with the jump intensity $\lambda_{\max}$ and the flow $S_{\lambda}$ in place of $S$. Taking into account the above observation, it is evident that such a modified system still satisfies conditions \eqref{constants_ass}- \eqref{A2} and \eqref{A3}-\eqref{A5}. Consequently, keeping in mind \eqref{eq:conv_lambda}, we can conclude that there exist $q_{\lambda_{\max}}\in(0,1)$ and $C_{\lambda_{\max},\mu} \in\mathbb{R}_+$ such that, for any $\lambda\in[\lambda_{\min},\lambda_{\max}]$, we have
\begin{align*}
\left\|\mu P^n_{\lambda}-\mu_{\lambda}^*\right\|_{FM}
=\left\|\mu\widetilde{P}_{\lambda_{\max}}^n - \mu_{\lambda}^* \right\|_{FM} 
\leq C_{\lambda_{\max},\mu} \,q_{\lambda_{\max}}^n\;\;\;\text{for any}\;\;\;\mu\in\mathcal{M}_{1,1}(X).
\end{align*}

In the case where $\alpha>0$, the proof is similar to the one conducted above (except that this time we substitute $t:=\lambda_{\min}\lambda^{-1} u$), so we omit~it.
\end{proof}

\section{Main results}\label{sec:main}

Before we formulate and prove the main theorems of this paper, let us first quote the result provided in \cite[Theorem 7.11]{rudin}.

\begin{lemma}\label{lem:rudin}
Let $(Y,\varrho)$ and $(Z,d)$ be some metric spaces, and let $U$ be an arbitrary subset of $Y$. Suppose that $(f_n)_{n\in\mathbb{N}_0}$ is a sequence of functions, defined on $E$, with values in $Z$, which converges uniformly on $E$ to some function $f:E\to Z$. Further, let $\bar{y}$ be a limit point of $E$, and assume that
\begin{align*}
a_n:=\lim_{y\to \bar{y}}f_n(y)
\end{align*}
exists and is finite for every ${n\in\mathbb{N}_0}$. 
Then, $f$ has a finite limit at $\bar{y}$, and the sequence $(a_n)_{n\in\mathbb{N}_0}$ converges to it, that is,
\begin{align*}
\lim_{n\to\infty}\left(\lim_{y\to \bar{y}}f_n(y)\right)=\lim_{y\to \bar{y}}\left(\lim_{n\to\infty}f_n(y)\right).
\end{align*}
\end{lemma}

We are now in a position to state the result concerning the continuous dependence of an invariant measure $\mu^*_{\lambda}$ of $P_{\lambda}$ on the parameter $\lambda$. In the proof we will refer to the lemmas provided in Section \ref{sec:properties}, as well as to Lemma \ref{lem:rudin}.

\begin{theorem}\label{thm:mu}
Suppose that conditions \eqref{A1}-\eqref{A5} hold with constants satisfying \eqref{constants_ass}, and, for any $\lambda \in [\lambda_{min}, \lambda_{max}]$, let $\mu_{\lambda}^*$ stand for the unique invariant probability measure measure of $P_{\lambda}$. Then, for every $\bar{\lambda}\in[\lambda_{\min},\lambda_{\max}]$, we have
$\mu_{\lambda}^* \stackrel{w}{\to} \mu_{\bar{\lambda}}^*$, as $\lambda\to\bar{\lambda}$.
\end{theorem}
\begin{proof}
Let  $\bar{\lambda}\in[\lambda_{\min},\lambda_{\max}]$. Due to Lemma \ref{lem:uniform_conv}, we know that, for every $\mu\in\mathcal{M}_1(X)$ and any  \hbox{$\lambda\in [\lambda_{\min}, \lambda_{\max}]$}, the sequence $(\mu P_{\lambda}^n)_{n\in\mathbb{N}_0}$ converges weakly to $\mu^*_{\lambda}$, as $n\to\infty$, and the convergence is uniform with respect to $\lambda$.

Further, Lemma \ref{lem:jointly_cnt} yields that $(\bar{\alpha},\infty)\times \mathcal{M}_1(X)\ni(\lambda,\mu)\mapsto \mu P_{\lambda}\in\mathcal{M}_1(X)$ is jointly continuous, provided that $\mathcal{M}_1(X)$ is equipped with the topology induced by the Fortet-Mourier norm. Hence, for any $\mu\in\mathcal{M}_1(X)$ and any $n\in\mathbb{N}_0$, it follows that $\mu P_{\lambda}^n$ converges weakly to $\mu P_{\bar{\lambda}}^n$, as $\lambda\to\bar{\lambda}$. 
Finally, according to Lemma \ref{lem:rudin}, we get
\begin{align*}
\lim_{\lambda\to\bar{\lambda}}\left\langle f,\mu_{\lambda}^*\right\rangle
&=\lim_{\lambda\to\bar{\lambda}}\left(\lim_{n\to\infty}\left\langle f,\mu P_{\lambda}^n\right\rangle\right)
=\lim_{n\to\infty}\left(\lim_{\lambda\to\bar{\lambda}}\left\langle f,\mu P_{\lambda}^n\right\rangle\right)
=\lim_{n\to\infty}\left\langle f,\mu P_{\bar{\lambda}}^n\right\rangle 
=\left\langle f,\mu_{\bar{\lambda}}^*\right\rangle
\end{align*}
for any $f\in BC(X)$ and any $\mu\in\mathcal{M}_1(X)$, which completes the proof.
\end{proof}

In the final part of the paper we will study the properties of the semigroup of Markov operators $(P_{\lambda}(t))_{t\in\mathbb{R}_+}$, defined by \eqref{def:P(t)}. I order to apply the relevant results of \cite{dawid}, in what follows, we additionally assume that the measure $\vartheta$, given on the set $\Theta$, is finite. Then, according to \cite[Theorem 4.4]{dawid}, for any $\lambda>0$, there is a one-to-one correspondence between invariant measures of the operator $P_\lambda$ and those of the semigroup $(P_{\lambda}(t))_{t\in\mathbb{R}_+}$. Furthermore, 
if $\mu^*_{\lambda}\in\mathcal{M}_1(X)$ is a unique invariant measure of $P_{\lambda}$, then $\nu^*_{\lambda}:=\mu_{\lambda}^*G_{\lambda}\in\mathcal{M}_1(X)$, where
\begin{align*}
\mu G_{\lambda}(A)=\int_X\int_0^{\infty}\lambda e^{-\lambda t}\mathbbm{1}_A\left(S(t,x)\right)\,dt\,\mu(dx)\;\;\;\text{for any}\;\;\;\mu\in\mathcal{M}_1(X),\;\;\;A\in\mathcal{B}_X,
\end{align*}
is a unique invariant measure of $(P_{\lambda}(t))_{t\in\mathbb{R}_+}$. 

The main result concerning the continuous-time model, which is formulated and proven below, ensures the continuity of the map $\lambda\mapsto\nu^*_{\lambda}$.

\begin{theorem}\label{thm:main}
Suppose that conditions \eqref{A1}-\eqref{A5} hold with constants satisfying \eqref{constants_ass}, and, for any $\lambda \in [\lambda_{min}, \lambda_{max}]$, let $\mu_{\lambda}^*$ stand for the unique invariant probability measure measure of $P_{\lambda}$. Further, assume that $\Theta$ is endowed with a finite Borel measure $\vartheta$. Then, for every $\bar{\lambda}\in[\lambda_{\min},\lambda_{\max}]$, we have
$\nu_{\lambda}^* \stackrel{w}{\to} \nu_{\bar{\lambda}}^*$, as $\lambda\to\bar{\lambda}$.
\end{theorem}

\begin{proof}
Let $\bar{\lambda}\in[\lambda_{\min},\lambda_{\max}]$, and let $f\in BL(X)$ be such that $\|f\|_{BL}\leq 1$. For any \linebreak$\lambda\in[\lambda_{\min},\lambda_{\max}]$, we have
\begin{align*}
\left\langle f,\nu_{\lambda}^*\right\rangle=\left\langle f,\mu_{\lambda}^*G_{\lambda}\right\rangle
=\int_X\int_0^{\infty}\lambda e^{-\lambda t}f\left(S(t,x)\right)\,dt\,\mu_{\lambda}^*(dx),
\end{align*}
whence 
\begin{align*}
\begin{aligned}
\left|\left\langle f,\nu_{\lambda}^*-\nu^*_{\bar{\lambda}}\right\rangle\right|
\leq &\int_0^{\infty}\left|\lambda e^{-\lambda t}-\bar{\lambda}e^{-\bar{\lambda}t}\right|\,dt
+
\left|\int_0^{\infty}\bar{\lambda}e^{-\bar{\lambda}t}\left\langle f\circ S(t,\cdot),\mu^*_{\lambda}-\mu^*_{\bar{\lambda}}\right\rangle\,dt\right|.
\end{aligned}
\end{align*}
Note that, due to \eqref{A2}, $f\circ S(t,\cdot)\in BL(X)$ and $\|f\circ S(t,\cdot)\|_{BL}\leq 1+Le^{\alpha t}$, which implies that 
\begin{align*}
\begin{aligned}
\left|\int_0^{\infty}\bar{\lambda}e^{-\bar{\lambda}t}\left\langle f\circ S(t,\cdot),\mu^*_{\lambda}-\mu^*_{\bar{\lambda}}\right\rangle\,dt\right|
&\leq \left\|\mu_{\lambda}^*-\mu_{\bar{\lambda}}^*\right\|_{FM}\int_0^{\infty} \bar{\lambda} e^{-\overline{\lambda} t}\left(1+Le^{\alpha t}\right) \,dt \\
&= \left\|\mu_{\lambda}^*-\mu_{\bar{\lambda}}^*\right\|_{FM}\left(
 1+\frac{L\bar{\lambda}}{\bar{\lambda}-\alpha}\right).
\end{aligned}
\end{align*}
Combining this and Lemma \ref{lem:mvt}, finally gives
\begin{align*}
\left\|\nu_{\lambda}^*-\nu_{\bar{\lambda}}^*\right\|_{FM} 
\leq \left|\lambda-\bar{\lambda}\right|
\left(\frac{1}{\lambda}+\frac{1}{\bar{\lambda}}\right)+c \left\|\mu_{\lambda}^*-\mu_{\bar{\lambda}}^*\right\|_{FM}
\end{align*}
with $c:=1+L\bar{\lambda}(\bar{\lambda}-\alpha)^{-1}$.  Hence, referring to the assertion of Theorem \ref{thm:mu}, we obtain 
\begin{align*}
\lim_{\lambda\to\bar{\lambda}}\left\|\nu_{\lambda}^*-\nu_{\bar{\lambda}}^*\right\|_{FM}=0,
\end{align*}
and the proof is completed.
\end{proof}

\section{Acknowledgements}
The work of Hanna Wojew\'odka-\'Sci\k{a}\.zko has been supported by the National Science Centre of Poland, grant number 2018/02/X/ST1/01518.

\bibliography{references}

\begin{thebibliography}{10}

\bibitem{alkurdi}
T.~Alkurdi, S.C. Hille, and O.~Van~Gaans.
\newblock Persistence of stability for equilibria of map iterations in banach
  spaces under small perturbations.
\newblock {\em Potential Anal.}, 42(11):175--201, 2015.

\bibitem{bogachev}
V.I. Bogachev.
\newblock {\em Measure Theory}.
\newblock Springer-Verlag, Berlin, 2007.

\bibitem{costa}
O.L.V. Costa and F.~Dufour.
\newblock Stability and ergodicity of piecewise deterministic \text{Markov}
  processes.
\newblock {\em SIAM J.~Control Optim.}, 47(2):1053--1077, 2008.

\bibitem{cl}
M.~Crandall and T.~Ligget.
\newblock Generation of semigroups of nonlinear transformations on general
  \text{B}anach spaces.
\newblock {\em Amer. J. Math.}, 93:265--298, 1971.

\bibitem{clt_chw}
D.~Czapla, K.~Horbacz, and H.~Wojew\'odka.
\newblock A useful version of the central limit theorem for a general class of
  {M}arkov chains.
\newblock {\em arXiv:1804.09220v2}, 2018.

\bibitem{dawid}
D.~Czapla, K.~Horbacz, and H.~Wojew\'odka.
\newblock Ergodic properties of some piecewise-deterministic {M}arkov process
  with application to gene expression modelling.
\newblock {\em To appear in Stochastic Process. Appl., doi:
  10.1016/j.spa.2019.08.006}, 2019.

\bibitem{lil_chw}
D.~Czapla, K.~Horbacz, and H.~Wojew\'odka-\'Sci\k{a}\.zko.
\newblock The {S}trassen invariance principle for certain non-stationary
  {M}arkov-{F}eller chains.
\newblock {\em arXiv:1810.07300v2}, 2018.

\bibitem{asia}
D.~Czapla and J.~Kubieniec.
\newblock Exponential ergodicity of some {M}arkov dynamical systems with
  application to a {P}oisson driven stochastic differential equation.
\newblock {\em Dyn. Syst.}, 34(1):130--156, 2019.

\bibitem{davis}
M.H.A. Davis.
\newblock Piecewise-deterministic \text{Markov} processes: a general class of
  non-diffusion stochastic models.
\newblock {\em J.~Roy. Statist. Soc. Ser. B}, 46(3):353--388, 1984.

\bibitem{vector_measures}
J.~Diestel and Jr. Uhl, J.J.
\newblock {\em Vector measures}.
\newblock American Mathematical Society, Providence, R.I., 1977.

\bibitem{dudley_baire}
R.M. Dudley.
\newblock Convergence of baire measures.
\newblock {\em Stud. Math.}, 27:251--268, 1966.

\bibitem{costa2000}
F.~Dufour and O.L.V. Costa.
\newblock Stability of piecewise-deterministic \text{Markov} processes.
\newblock {\em SIAM J.~Control Optim.}, 37(5):1483--1502, 2000.

\bibitem{hille_evers}
J.~Evers, S.C. Hille, and A.~Muntean.
\newblock Mild solutions to a measure-valued mass evolution problem with flux
  boundary conditions.
\newblock {\em J. Diff. Equ.}, 259:1068--1097, 2015.

\bibitem{hille_lyczek}
P.~Gwiazda, S.C. Hille, K.~\L{}yczek, and A.~\'Swierczewska-Gwiazda.
\newblock Differentiability in perturbation parameter of measure solutions to
  perturbed transport equation.
\newblock {\em To appear in Kinet. Relat. Mod.}, currenlty in
  ArXiv:18l06.00357, 2019.

\bibitem{hhs}
S.C. Hille, K.~Horbacz, and T.~Szarek.
\newblock Existence of a unique invariant measure for a~class of equicontinuous
  {M}arkov operators with application to a stochastic model for
  an~autoregulated gene.
\newblock {\em Ann. Math. Blaise Pascal}, 23(2):171--217, 2016.

\bibitem{hhsw2}
S.C. Hille, K.~Horbacz, T.~Szarek, and H.~Wojew\'{o}dka.
\newblock Law of the iterated logarithm for some {Markov} operators.
\newblock {\em Asymptotic Anal.}, 97(1-2):91--112, 2016.

\bibitem{hhsw}
S.C. Hille, K.~Horbacz, T.~Szarek, and H.~Wojew\'{o}dka.
\newblock Limit theorems for some \text{M}arkov chains.
\newblock {\em J. Math. Anal. Appl.}, 443(1):385--408, 2016.

\bibitem{klo}
T.~Komorowski, C.~Landim, and S.~Olla.
\newblock {\em Fluctuations in \text{M}arkov processes. Time symmetry and
  martingale approximation}.
\newblock Springer-Verlag, Heidelberg, 2012.

\bibitem{lm}
A.~Lasota and M.C. Mackey.
\newblock Cell division and the stability of cellular populations.
\newblock {\em J.~Math. Biol.}, 38:241--261, 1999.

\bibitem{tyran}
M.C. Mackey, M.~Tyran-Kami\'{n}ska, and R.~Yvinec.
\newblock Dynamic behaviour of stochastic gene expression models in the
  presence of bursting.
\newblock {\em {SIAM} J. Appl. Math.}, 73(5):1830--1852, 2013.

\bibitem{riedler_ea}
M.G. Riedler, M.~Thieullen, and G.~Wainrib.
\newblock Limit theorems for infinite-dimensional piecewise deterministic
  \text{M}arkov processes. applications to stochastic excitable membrane
  models.
\newblock {\em Electron. J. Probab.}, 17(55):1--48, 2012.

\bibitem{rudin}
W.~Rudin.
\newblock {\em Principles of mathematical analysis}.
\newblock McGraw-Hill, Inc., New York, 1976.

\bibitem{weaver}
N.~Weaver.
\newblock {\em Lipschitz Algebras}.
\newblock World Scientific Publishing Co. Pte Ltd., Singapore, 1999.

\bibitem{hw}
H.~Wojew\'odka.
\newblock Exponential rate of convergence for some \text{Markov} operators.
\newblock {\em Statist. Probab. Lett.}, 83(10):2337--2347, 2013.

\bibitem{worm}
D.T.H. Worm.
\newblock Semigroups on spaces of measures.
\newblock {\em PhD. thesis, Leiden University, The Netherlands}, Available at:
  www.math.leidenuniv.nl/nl/theses/PhD/, 2010.

\end{thebibliography}
\bibliographystyle{plain}
\end{document}